\documentclass[a4paper,10pt]{article}
\usepackage{textcomp}
\usepackage{graphicx}
\usepackage{color}
\usepackage[colorlinks]{hyperref}
\usepackage{lscape}
\usepackage[centertags]{amsmath}
\usepackage{amssymb}
\usepackage{amsthm}
\usepackage{newlfont}
\usepackage{amsfonts}
\theoremstyle{plain}
\newtheorem{thm}{Theorem}
\newtheorem{cor}{Corollary}
\newtheorem{lem}{Lemma}
\newtheorem{prop}{Proposition}

\newtheorem{defn}{Definition}

\newcommand{\sub}{\subseteq}

\begin{document}

\begin{center}
\Large \textbf{The structure of graphs with forbidden $C_4$, $\overline{C}_4$, $C_5$, chair and co-chair}
\end{center}
\begin{center}
Salman GHAZAL\footnote{\noindent Department of Mathematics, Faculty of Sciences I, Lebanese University, Hadath, Beirut, Lebanon.\\
                       E-mail: salmanghazal@hotmail.com\\

                       }
\end{center}
\vskip1cm
\begin{abstract}
 We find the structure of graphs that have no $C_4$, $\overline{C}_4$, $C_5$, chair and co-chair as induced subgraphs.
\end{abstract}

\begin{section}{Introduction}

\par \hskip0.6cm In this paper, graphs are finite and simple. The vertex set and edge set of a graph $G$ are  denoted by $V(G)$ and $E(G)$ respectively. Two edges of a graph $G$ are said to be adjacent if they have a common endpoint and two vertices $x$ and $y$ are said to be adjacent if $xy$ is an edge of $G$. The \emph{neighborhood} of a vertex $v$ in a graph $G$, denoted by $N_G(v)$, is the set of all vertices adjacent to $v$ and its \emph{degree} is $d_G(v)=|N_G(v)|$. We omit the subscript if the graph is clear from the context. For two set of vertices $U$ and $W$ of a graph $G$, let $E[U, W]$ denote the set of all edges in the graph $G$ that joins a vertex in $U$ to a vertex in $W$. A graph is empty if it has no edges. For $A\sub V(G)$, $G[A]$ denotes the sub-graph  of $G$ induced by $A$. If $G[A]$ is an empty graph, then $A$ is called a stable. While, if $G[A]$ is a complete graph, then $A$ is called a clique set, that is any two distinct vertices in $A$ are adjacent. The complement graph of $G$ is denoted by $\overline{G}$ and defined as follows: $V(G)=V(\overline{G})$ and $xy\in E(\overline{G})$ if and only if $xy\notin E(G)$. \\

A graph $H$ is called forbidden subgraph of $G$ if $H$ is not (isomorphic to) an induced subgraph of $G$.\\

A cycle on $n$ vertices is denoted by $C_n=v_1v_2...v_nv_1$ while a path on $n$ vertices is denoted by $P_n=v_1v_2...v_n$. A chair is any graph on 5  distinct vertices $x,y,z,t,v$ with exactly 5 edges $xy,yz,zt$ and $zv$. The co-chair or $\overline{chair}$ is the complement of a chair (see the below figure).

\begin{center}
\includegraphics[width=8cm, height=7cm]{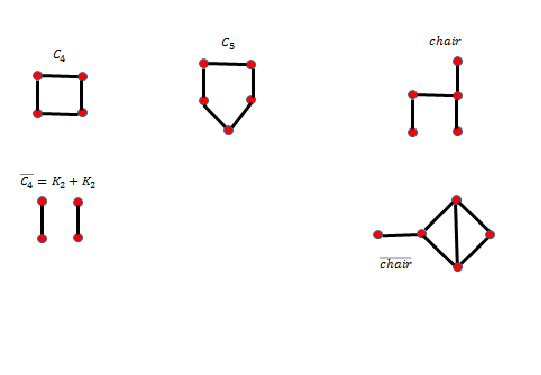}
\end{center}

Many graphs encountered in the study of graph theory are characterized by configurations or subgraphs they contain. However, there are occasions where it is easier to characterize graphs by sub-graphs or induced sub-graphs they do not contain. For example, trees are the connected graph without (induced) cycles. Bipartite graphs are those without (induced) odd cycles (\cite{Konig}). Split graphs are those without induced $C_4$, $\overline{C}_4$ and $C_5$. Line graphs are characterized by the absence of only nine particular graphs as induced sub-graph (see \cite{linech}). Perfect graphs are characterized by $C_{2n+1}$ and $\overline{C}_{2n+1}$ being forbidden, for all $n\geq 2$ (see \cite{perfectch}).The purpose of this paper is to find the structure of graphs such that $C_4$, $\overline{C}_4$, $C_5$, chair and co-chair are forbidden subgraphs.
\end{section}

\begin{section}{Preliminary Definitions and Theorems}

\begin{defn}
A graph $G$ is a called a split graph if its vertex set is the disjoint union of a stable set $S$ and a clique set $K$. In this case, $G$ is called an $\{S$, $K\}$-split graph.
\end{defn}

If $G$ is an $\{S$, $K\}$-split graph and $\forall s\in S$, $\forall x\in K$ we have $sx\in E(G)$, then $G$ is called a complete split graph.\\

If $G$ is an $\{S$, $K\}$-split graph and $E[S,K]$ forms a perfect matching of $G$, then $G$ is called a perfect split graph.

\begin{thm} (F\"{o}ldes and Hammer \cite{splitch})
$G$ is a split graph if and only if $C_4$, $\overline{C}_4$ and $C_5$ are forbidden subgraphs of $G$.
\end{thm}

\begin{defn} (\cite{thresholdch})
A threshold graph $G$ can be defined as follows:
\begin{description}
\item[1) ] $V(G)=\displaystyle\bigcup_{i=1}^{n+1}(X_i\cup A_{i-1})$, where the $A_i$'s and $X_i$'s are pair-wisely disjoint sets.
\item[2) ] $K:=\displaystyle\bigcup_{i=1}^{n+1}X_i$ is a clique and the $X_i$'s are nonempty, except possibly $X_{n+1}$.
\item[3) ] $S:=\displaystyle\bigcup_{i=0}^{n}A_{i}$ is a stable set and the $A_i$'s are nonempty, except possibly $A_0$.
\item[4) ] $\forall 1\leq j\leq i\leq n$, $G[A_i\cup X_j]$ is a complete split graph.
\item[5)] The only edges of $G$ are the edges of the subgraphs mentioned above.
\end{description}

In this case, $G$ is called an $\{S,$ $K\}$-threshold graph.

\end{defn}

\begin{thm}(Hammer and Chv\`{a}tal \cite{thresholdch})
$G$ is a threshold graph if and only if $C_4$, $\overline{C}_4$ and $P_4$ are forbidden subgraphs of $G$.
\end{thm}

\end{section}

\begin{section}{Main Results}

\begin{lem}
  Suppose that $C_4$, $\overline{C}_4$, $C_5$, chair and co-chair are forbidden subgraphs of $G$. If the path $mbb'm'$ is an induced subgraph of $G$, then:

  $$N(m)-\{b\}=N(m')-\{b'\}$$ and $$N(b)-\{m\}=N(b')-\{m'\}.$$
\end{lem}

\begin{proof}

 Since $C_4$, $\overline{C}_4$ and $C_5$ are forbidden, then $G$ is an $\{S$, $K\}$-split graph for some stable set $S$ and a clique set $K$. Since $mbb'm'$ is an induced subgraph of $G$, then $m$, $m'\in S$ and $b$, $b'\in K$.\\

 Assume that there is $x\in N(m)-\{b\}$ but $x\notin N(m')-\{b'\}$. Since $xm$ is an edge of $G$ and $S$ is stable, then we must have $x\in K$. But $K$ is a clique, then $x$ is adjacent to $b$ and $b'$. Thus $G[\{x,m,b,b',m'\}]$ is a co-chair. Contradiction. So  $N(m)-\{b\}\subseteq N(m')-\{b'\}$. By symmetry, $N(m')-\{b'\}\subseteq N(m)-\{b\}$. Thus $N(m)-\{b\}=N(m')-\{b'\}$.\\

 Assume that there is $x\in N(b)-\{m\}$ but  $x\notin N(b')-\{m'\}$. Suppose that $x\in S$. Then $G[\{x,m,b,b',m'\}]$ is a chair. Contradiction. Thus $x\in K$. But $K$ is a clique. Whence $x\in N(b')\{m'\}$. Thus $N(b)-\{m\}\subseteq N(b')-\{m'\}$. By symmetry, $N(b')-\{m'\}\subseteq N(b)-\{m\}$. Therefore $N(b)-\{m\}=N(b')-\{m'\}$.

\end{proof}

\begin{prop}
  If $P_4$ is a forbidden subgraph of an  $\{S,$ $K\}$-split graph $G$, then $G$ is an $\{S,$ $K\}$-threshold graph.
\end{prop}

\begin{proof}
We prove this by induction on the number of vertices of $G$. This is clearly true for small graphs. Suppose that $P_4$ is a forbidden subgraph of an  $\{S,$ $K\}$-split graph $G$. It is clear that $G$ is a threshold graph. We have to prove that $G$ is $\{S,$ $K\}$-threshold graph. Let $x\in K$ be a vertex with minimum degree in $G$, that is $d_G(x)=min\{d_G(y);$ $y\in K\}$ and $G':=G-x$ be the graph induced by the vertices of $G$ except $x$ (If $K=\phi$, then the statement is true). Then $P_4$ is a forbidden subgraph of the  $\{S,$ $K-\{x\}\}$-split graph $G'$. By the induction hypothesis, $G'$ is an $\{S,$ $K-\{x\}\}$-threshold graph. We follow the notations in Definition 2. Assume that $\exists a\in S-A_n$ such that $ax\in E(G)$. Let $x_n\in X_n$. Since $d(x_n)\geq d(x)$, then there is $a_n\in A_n$ such that $a_nx_n\in E(G)$ but $a_nx\notin E(G)$. Then $axx_na_n$ is an induced $P_4$ in $G$. Contradiction. Thus we may suppose that $N(x)\cap S\subseteq A_n$. If $N(x)\cap A_n = \phi$, then we add $x$ to $X_{n+1}$. If $N(x)\cap A_n =A_n$, then we add $x_n$ to $X_n$. Otherwise $\phi\subsetneq N(x)\cap A_n\subsetneq A_n$. In this case we do the following: remove from $A_n$ the element of $N(x)\cap A_n$, create $A_{n+1}=N(x)\cap A_n$, remove the elements of $X_{n+1}$ to the new set $X_{n+2}$ and add $x$ to $X_{n+1}$ (so that the new $X_{n+1}=\{x\}$). Then $G$ is $\{S,$ $K\}$-threshold graph
\end{proof}

\begin{defn} \label{gcdef}
A graph $G$ is called a comb if:
\begin{description}
  \item[1)] $V(G)$ is disjoint union of sets $A_0,...,A_n, M_1,...,M_l,X_1,....,X_{n+1}, Y_2,...,Y_{l+2}$. Let $Y_1=X_1$ (These sets are called the sets of the comb $G$).
  \item[2)] $S:=A\cup M$ is a stable set, where $M=\displaystyle\bigcup_{i=1}^{l}M_{i}$ and $A=\displaystyle\bigcup_{i=0}^{n}A_{i}$
  \item[3)] $K:=X\cup Y$ is a clique, where $X=\displaystyle\bigcup_{i=1}^{n+1}X_{i}$ and $Y=\displaystyle\bigcup_{i=1}^{l+2}Y_{i}$.
  \item[4)] $\forall 1\leq j\leq i\leq n$, $G[A_i\cup X_j]$ is a complete split graph.
  \item[5)]$G[A\cup Y]$ is a complete split graph.
  \item[6)]$\forall 1\leq i\leq l$, $G[Y_i\cup M_i]$ is a perfect split graph.
  \item[7)] $\forall 1\leq i <j \leq l$, $G[Y_j\cup M_i]$ is a complete split graph.
  \item[8)] $\exists 1\leq k_0\leq l$, $\forall i\leq k_0$, $G[Y_{l+1}\cup M_i]$ is a complete split graph.
  \item[9)] $X_{n+1}, Y_{l+2}, Y_{l+1}, M_l$ and $A_0$ are the only possibly empty sets.
  \item[10)] The only edges of $G$ are the edges of the subgraphs mentioned above.
\end{description}

In this case, we say that $G$ is an $\{S,$ $K\}$-comb.
\end{defn}

\begin{lem}
  Every $\{S,$ $K\}$-threshold graph is an  $\{S,$ $K\}$-comb.
\end{lem}

\begin{proof}
  Let $G$ be an $\{S,$ $K\}$-threshold graph defined as in Definition 2. Following the notations in Definition 3, we take $l=1$ and $M_l=Y_{l+1}=Y_{l+2}=\phi$. This shows that $G$ is an $\{S,$ $K\}$-comb.
\end{proof}

\begin{thm}
  If chair and co-chair are forbidden subgraphs of an $\{S,$ $K\}$-split graph $G$ , then $G$ is an $\{S,$ $K\}$-comb.
\end{thm}

\begin{proof}
We prove the statement by induction on the number of vertices. The statement is true for small graphs. Suppose that chair and co-chair are forbidden subgraphs of an $\{S,$ $K\}$-split graph $G$. If $P_4$ is also a forbidden subgraph of $G$, then $G$ is an  $\{S,$ $K\}$-threshold graph, and hence, $G$ is an  $\{S,$ $K\}$-comb. So we may suppose that $G$ contains an induced path $abb'a'$. Then $N(a)-\{b\}=N(a')-\{b'\}$ and $N(b)-\{a\}=N(b')-\{a'\}$. Let $S'=S-a'$, $K'=K-b'$ and $G'=G[S'\cup K']$. Then chair and co-chair are forbidden subgraphs of the $\{S',$ $K'\}$-split graph $G'$. Then $G'$ is an $\{S',$ $K'\}$-comb with $S'=A\cup M$ and $K'=X\cup Y$ (we follow the notations as in Definition 3).\\

If $a\in S'$ and $b\in K'$, then we add $a'$ to the set of the comb $G'$ that contains $a$ and $b'$ to the set of the comb $G'$ that contains $b$. Thus $G$ is $\{S,$ $K\}$-comb.\\

Otherwise, $a\in K$ while $b\in S$. First we suppose that $n\geq 1$. Then there is $x\in A_1$ because $A_1\neq \phi$. We have the following cases:\\

case 1: assume that $a\in Y$ and $b\in M$. Then $xabb'a'x$ is an induced $C_5$ in $G$. Contradiction.\\

case 2: assume that $a\in X_i$ and $b\in A_j$. Then by definition of comb, we have $i\leq j$. Then $xabb'a'x$ is an induced $C_5$ in $G$. Contradiction. So $i=j$. Assume that there is $y\in \displaystyle\bigcup_{t=i}^{n}A_t-\{b\}$. Then $yaba'b'y$ is an induced $C_5$ in $G$. Contradiction. Thus we must have $i=n$ and $A_i=A_n=\{b\}$. Assume that there is $y\in X_{n+1}$. Then $yaba'b'y$ is an induced $C_5$ in $G$. Contradiction. Thus we must have $X_{n+1}=\phi$. In this case, we do the following: remove $a$ from $X_n$ and add it to $A_n$, remove $b$ from $A_n$ and add it to $X_n$, add $b'$ to $X_{n+1}$, create $A_{n+1}=\{a'\}$ and $X_{n+2}=\phi$. Thus $G$ is an $\{S,$ $K\}$-comb.\\

case 3: assume that $a\in X_i$ and $b\in M_j$. Then by the definition of a comb, we must have $i=1=j$. But this is already discussed in case 1, because $X_1=Y_1$.\\

case 4: Assume that $a\in Y_i$ and $b\in A_j$. The case when $i=1$ is already discussed in case 2. So we may assume that $i>1$. Let $y\in M_1$. Then $yaba'b'y$ is an induced $C_5$ in $G$. Contradiction.\\

Second, suppose that $n=0$. That is $A=A_0$ and so there is no $A_1$ and no $X_2$. We have the following cases:\\

case 1: Assume that $a\in Y_i$ and $b\in M_i$. If $i>1$ or $Y_i\neq \{b\}$, then there is $c\in \displaystyle\bigcup_{t=1}^{i}A_t-\{a\}$. Then $cabb'a'c$ is an induced $C_5$ in $G$. Contradiction. Thus $i=1$ and $Y_1=\{a\}$. Hence $M_1=\{b\}$. We can do the following: remove $a$ from $Y_1$ and add it to $M_1$, remove $b$ from $M_1$ and add it $Y_1$, add $b'$ to $Y_1$ and add $a'$ to $M_1$. Thus $G$ is an  $\{S,$ $K\}$-comb.\\

case 2: Assume that $a\in Y_i$ and $b\in M_j$ with $i>j$. There exist $c\in Y_j$ such that $cb$ is an edge of $G$. If there is $y\in N_{G'}(a)-N_{G'}(b)$, then $yabb'a'y$ is an induced $C_5$ in $G$. Contradiction. Thus, we must have $j=1$, $Y_1=\{c\}$, $M_1=\{b\}$, $i=2$ and $M_2=\phi$. We can do the following: remove $a$ from $Y_2$ and add it to $M_1$, remove $b$ from $M_1$ and add it $Y_1$ and remove $c$ from $Y_1$ and add it to $Y_2$. Thus $G$ is an  $\{S,$ $K\}$-comb.\\

case 3: $a\in Y_i$ and $b\in M_j$ with $i<j$. This case is impossible by the definition of the comb.

\end{proof}

\begin{cor}
$G$ is a comb if and only if $C_4$, $\overline{C}_4$, $C_5$, chair and co-chair are forbidden subgraphs of $G$.
\end{cor}

\begin{proof}
The necessary condition is obvious by the definition of a comb. For the sufficient condition it is enough to note that the statement $C_4$, $\overline{C}_4$, $C_5$, chair and co-chair are forbidden subgraphs of $G$ is equivalent to the statement that $G$ is a split graph and chair and co-chair are forbidden subgraphs of $G$.\\
\end{proof}

\begin{cor}
  $G$ is a comb if and only if $\overline{G}$ is a comb.
\end{cor}

\begin{proof}
Enough to note that the complement of $C_4$, $\overline{C}_4$, $C_5$, chair and co-chair are $\overline{C}_4$, $C_4$, $C_5$, co-chair and chair.
\end{proof}

\begin{cor}
  $G$ is a comb if and only if every induced subgraph of $G$ is a comb.
\end{cor}

\end{section}


\begin{thebibliography}{2}

\bibitem[1] {splitch}  S. F\"{o}ldes, P.L. Hammer , \emph{Split graphs}, Proceedings of the Eighth Southeastern Conference on Combinatorics, Graph Theory and Computing (Louisiana State Univ., Baton Rouge, La., 1977), Congressus Numerantium XIX, Winnipeg: Utilitas Math., pp. 311–315

\bibitem[2]{thresholdch} P.L. Hammer, V, Chv\`{a}tal, \emph{Aggregation of inequalities in integer programming}, Annals of Discrete Mathematics 1 (1977), 145–162.

\bibitem[3] {perfectch} M. Chudnovsky, N. Robertson, P. Seymour, R. Thomas, \emph{The strong perfect graph theorem}, Annals of Mathematics 164 (1) (2006), 51–229.

\bibitem[4] {linech} L.W. Beineke, \emph{Characterizations of derived graphs}, Journal of Combinatorial Theory 9 (2) (1970), 129–135.
\bibitem [5] {Konig} D. K\"{o}, \emph{Theorie der endlichen und unendlechen Graphen}, Akademische Verllagsgesellschaft (1936) (reprinted Chelsea 1950).

\end{thebibliography}
\end{document}